\documentclass{amsart}

\usepackage{amssymb,latexsym,amsmath}
\usepackage{hyperref}
\usepackage{xcolor}
\hypersetup{colorlinks=true, linkcolor = blue, citecolor = red}

\newtheorem{theorem}{Theorem}[section]
\newtheorem{lemma}[theorem]{Lemma}
\newtheorem{proposition}[theorem]{Proposition}

\theoremstyle{definition}
\newtheorem{definition}[theorem]{Definition}

\theoremstyle{remark}
\newtheorem{remark}[theorem]{Remark}

\numberwithin{equation}{section}

\newcommand{\abs}[1]{\left|#1\right|}
\newcommand{\norm}[1]{\left\|#1\right\|}

\begin{document}

\title[Sobolev regularity]{Weighted Sobolev regularity of the Bergman projection on the Hartogs triangle}

\author{Liwei Chen}
\address{Department of Mathematics, The Ohio State University, Ohio 43210}
\email{chen.1690@osu.edu}


\date{\today}


\begin{abstract}
We prove a weighted Sobolev estimate of the Bergman projection on the Hartogs triangle, where the weight is some power of the distance to the singularity at the boundary. This method also applies to the $n$-dimensional generalization of the Hartogs triangle. 
\end{abstract}

\maketitle

\section{Introduction}

\subsection{Setup and Background}

Let $\Omega$ be a domain in $\mathbb{C}^n$. The set of square integrable holomorphic functions on $\Omega$, denoted by $A^2(\Omega)$, forms a closed subspace of the Hilbert space $L^2(\Omega)$. The Bergman projection associated to $\Omega$, is the orthogonal projection
\[
\mathcal{B}:L^2(\Omega)\to A^2(\Omega),
\]
which has an integral representation
\begin{equation}
\label{kernel}
\mathcal{B}(f)(z)=\int_{\Omega}B(z,\zeta)f(\zeta)\,d(\zeta),
\end{equation}
for all $f\in L^2(\Omega)$ and $z\in\Omega$. Here the function $B(z,\zeta)$ defined on $\Omega\times\Omega$ is the Bergman kernel.

The regularity of the Bergman projection $\mathcal{B}$ associated to $\Omega$ in $L^p(\Omega)$, $W^{k,p}(\Omega)$, and H\"{o}lder spaces are of particular interest. When $\Omega$ is bounded, smooth, and pseudoconvex with additional geometric condition on the boundary (e.g. strongly pseudoconvex), the regularity of $\mathcal{B}$ in these spaces have been intensively studied through the literature. See, for example, \cite{LS2} and references therein for details.

When $\Omega$ is non-smooth, there are relatively few results in considering the regularity of the Bergman projection. Even in $L^p(\Omega)$, we cannot expect the regularity to hold for all $p\in(1,\infty)$. If $\Omega$ is a simply connected planar domain, then the interval of $p$ for $\mathcal{B}$ to be $L^p$-bounded highly depends on the geometry of the boundary, see \cite{LS1}. If $\Omega$ is a non-smooth worm domain, then the interval of $p$ depends on the winding of the domain, see \cite{KP}. If $\Omega$ is an inflation of the unit disc by the norm square of a non-vanishing holomorphic function, then the interval of $p$ depends on the boundary behavior of the holomorphic function on the unit disc, see \cite{Z}.

\subsection{Results}

In this article, we consider the Sobolev regularity of the Bergman projection $\mathcal{B}$ on the Hartogs triangle $\mathbb{H}$, where the Hartogs triangle is defined as
\[
\mathbb{H}=\{(z_1,z_2)\in\mathbb{C}^2\,|\,\abs{z_1}<\abs{z_2}<1\}.
\]
The Hartogs triangle is a classical non-smooth domain in $\mathbb{C}^2$. It is well known that the boundary at $(0,0)$ is not even Lipschitz, and the topological closure of $\mathbb{H}$ does not possess a Stein neighborhood basis. In \cite{C1}, the $L^p$ regularity of $\mathcal{B}$ on $\mathbb{H}$ has been studied: the Bergman projection $\mathcal{B}$ is $L^p$-bounded if and only if $p\in(4/3,4)$. On the other hand, we have $\overline{z}_2\in W^{k,p}(\mathbb{H})$ for all non-negative integer $k$ and all $p\in[1,\infty]$, but $\mathcal{B}(\overline{z}_2)=c/z_2\notin W^{1,p}(\mathbb{H})$ for $p\ge2$, where $c$ is some non-zero constant. So we cannot expect to obtain regularity in the ordinary Sobolev spaces, nor for all $p\in(1,\infty)$.

A natural way to control the boundary behavior of singularity is the use of weights which measure the distance from the points near the boundary to the singularity at the boundary. Since on the Hartogs triangle we have $\abs{z_2}<\abs{z}<\sqrt{2}\abs{z_2}$, where $z = (z_1, z_2)\in\mathbb{H}$, it is reasonable to consider a weight of the form $\abs{z_2}^s$, for some $s\in\mathbb{R}$. On the other hand, based on the $L^p$ mapping property of the Bergman projection on $\mathbb{H}$ (see \cite{CZ}) and the Sobolev regularity of the weighted canonical solution operator of the $\overline{\partial}$-equation on $\mathbb{H}$ (see \cite{CS2}), it is also suggested to put a weight of the form $\abs{z_2}^s$ on the target space. Therefore, we consider the following weighted Sobolev spaces.

\begin{definition}
\label{Sobolev}
On the Hartogs triangle $\mathbb{H}$, for each $k\in\mathbb{Z}^+\cup\{0\}$, $s\in\mathbb{R}$, and $p\in(1,\infty)$, we define the \textit{weighted Sobolev space} by
\[
W^{k,p}(\mathbb{H},\delta^{s})=\{f\in L^1_{\text{loc}}(\mathbb{H})\,|\,\norm{f}_{k,p,s}<\infty\},
\]
where $\delta(z)=\abs{z_2}\approx\abs{z}$ and the norm is defined as
\[
\norm{f}_{k,p,s}=\Big(\int_{\mathbb{H}}\sum_{\abs{\alpha}\le k}\abs{D_{z,\overline{z}}^{\alpha}(f)(z)}^p\abs{z_2}^s\,dz\Big)^{\frac{1}{p}}.
\]
Here $\alpha=(\alpha_1,\alpha_2,\alpha_3,\alpha_4)$ is the multi-index running over all $\abs{\alpha}\le k$, and
\[
D_{z,\overline{z}}^{\alpha}=\frac{\partial^{\abs{\alpha}}}{\partial z_1^{\alpha_1}\partial z_2^{\alpha_2}\partial\overline{z}_1^{\alpha_3}\partial\overline{z}_2^{\alpha_4}}.
\]
We also denote the usual norm in the (unweighted) Sobolev space $W^{k,p}(\mathbb{H})$ by 
\[
\norm{f}_{k,p}=\Big(\int_{\mathbb{H}}\sum_{\abs{\alpha}\le k}\abs{D^{\alpha}_{z,\overline{z}}(f)(z)}^p\,dz\Big)^{\frac{1}{p}}.
\]
\end{definition}

With the definition above, we can state our main result as follows.

\begin{theorem}
\label{main}
The Bergman projection $\mathcal{B}$ on the Hartogs triangle $\mathbb{H}$ maps continuously from $W^{k,p}(\mathbb{H})$ to $W^{k,p}(\mathbb{H},\delta^{kp})$ for $p\in(4/3,4)$.

That is, for each $k\in\mathbb{Z}^+\cup\{0\}$ and $p\in(4/3,4)$, there exits a constant $C_{k,p}>0$, so that
\[
\norm{\mathcal{B}(f)}_{k,p,kp}\le C_{k,p}\norm{f}_{k,p},
\]
for any $f\in W^{k,p}(\mathbb{H})$.
\end{theorem}

\begin{remark}
Note that there is no loss of differentiability of $\mathcal{B}(f)$. This improves the previous result in \cite{CS2}. 
\end{remark}

\begin{remark}
Note that we have $\mathcal{B}(\overline{z}_2)=c/z_2\notin W^{k,p}(\mathbb{H},\delta^{kp})$ for $p\ge4$, where $c$ is some non-zero constant. So we cannot obtain regularity for $p\ge4$, unless we use more weights on the target space. Conversely, we can only obtain regularity for fewer values of $p$, if we use less weights on the target space.
\end{remark}

\subsection{Organization and Outline}

The idea of the proof of the main result is the following. In section \ref{sec2}, we start with an idea from \cite{CS2} to transfer $\mathbb{H}$ to the product model $\mathbb{D}\times\mathbb{D}^*$, as well as to transfer the differential operators $D^{\alpha}$ to the ones in new variables. From this, we focus on the integration over the punctured disc $\mathbb{D}^*$ in section \ref{sec3}. We then use an idea from \cite{Str1} to convert $D^{\alpha}$ acting on the Bergman kernel in the holomorphic component to the ones acting on the kernel in the anti-holomorphic part. The resulting differential operators can be written as a combination of tangential operators, and therefore, integration by parts applies to the smooth functions. Finally, in section \ref{sec4}, we apply the weighted $L^p$ estimates in \cite{C2} to our integral, and the resulting integral is majorized by the weighted $L^p$ norm of $D^{\alpha}(f)$. It will complete the proof if we approximate the weighted Sobolev functions by smooth functions and transfer the product model back to $\mathbb{H}$.

\subsection*{Acknowledgements}

The content of this paper is a part of the author's Ph.D.thesis at Washington University in St. Louis, see \cite{Chen}. The author would like to thank his thesis advisor Prof. S. G. Krantz for giving him a very interesting problem to work on and lots of guidance on his research. The author also wants to thank Prof. E. J. Straube for very helpful comments and suggestions on his work.

\section{Transfer to the Product Model}
\label{sec2}

\subsection{Transfer $\mathbb{H}$ to $\mathbb{D}\times\mathbb{D}^*$}

In view of Definition \ref{Sobolev}, we adopt the following notations.
\begin{definition}
\label{differential}
Let $\beta=(\beta_1,\beta_2)$ be a multi-index, we use the notations below to denote the differential operators
\[
D_{z}^{\beta}=\frac{\partial^{\abs{\beta}}}{\partial z_1^{\beta_1}\partial z_2^{\beta_2}}
\]
and
\[
D_{z_j,\overline{z}_j}^{\beta}=\frac{\partial^{\abs{\beta}}}{\partial z_j^{\beta_1}\partial\overline{z}_j^{\beta_2}}
\]
for $j=1,2$.
\end{definition}

From the result in \cite{C1}, we see that $\mathcal{B}(f)\in A^p(\mathbb{H})$ (the set of $L^p$ functions that are holomorphic), whenever $p\in(4/3,4)$ and $f\in L^p(\mathbb{H})$. So we can rewrite the weighted $L^p$ Sobolev norm of $\mathcal{B}(f)$ as

\begin{equation}
\label{eq1}
\norm{\mathcal{B}(f)}_{k,p,kp}^p=\sum_{\abs{\beta}\le k}\int_{\mathbb{H}}\abs{D_{z}^{\beta}(\mathcal{B}(f))(z)}^p\abs{z_2}^{kp}\,dz,
\end{equation}
where $\beta$ and $D_{z}^{\beta}$ are as in Definition \ref{differential}.

In order to transfer $\mathbb{H}$ to the product model, we first recall the transformation formula for the Bergman kernels.

\begin{proposition}
Let $\Omega_j$ be a domain in $\mathbb{C}^n$ and $B_j$ be its Bergman kernel on $\Omega_j\times\Omega_j$, $j=1,2$. Suppose $\Psi:\Omega_1\to\Omega_2$ is a biholomorphism, then for $(w,\eta)\in\Omega_1\times\Omega_1$ we have
\[
\det J_{\mathbb{C}}\Psi(w)B_2(\Psi(w),\Psi(\eta))\det\overline{J_{\mathbb{C}}\Psi(\eta)}=B_1(w,\eta).
\]
\end{proposition}

\begin{proof}
See, for example, \cite[Proposition 1.4.12]{K}.
\end{proof}

Now let us consider the biholomorphism
\[
\Phi:\mathbb{H}\to\mathbb{D}\times\mathbb{D}^*
\]
with its inverse
\[
\Psi:\mathbb{D}\times\mathbb{D}^*\to\mathbb{H},
\]
where
\[
\Phi(z_1,z_2)=(\frac{z_1}{z_2},z_2)\,\,\,\,\,\text{and}\,\,\,\,\,\Psi(w_1,w_2)=(w_1w_2,w_2).
\]
A simple computation shows $\det J_{\mathbb{C}}\Psi(w)=w_2$, for $w=(w_1,w_2)\in\mathbb{D}\times\mathbb{D}^*$. Therefore, by the proposition above, we have
\begin{equation}
\label{trans}
B(\Psi(w),\Psi(\eta))=\frac{1}{w_2\overline{\eta}_2}\cdot\frac{1}{(1-w_1\overline{\eta}_1)^2}\cdot\frac{1}{(1-w_2\overline{\eta}_2)^2},
\end{equation}
where $B$ is the Bergman kernel on $\mathbb{H}\times\mathbb{H}$ as in \eqref{kernel} and $(w,\eta)\in\mathbb{D}\times\mathbb{D}^*\times\mathbb{D}\times\mathbb{D}^*$.

\subsection{Transfer the Differential Operators}

We next need to transfer the differential operators $D_{z}^{\beta}$ to the ones in the new variable $w$. We need a lemma.

\begin{lemma}
\label{lem}
Under the above biholomorphism $\Phi(z)=w$, for each $\beta$ let $m=\abs{\beta}$, we have
\begin{equation}
\label{DztoDw}
D_{z}^{\beta}=\sum_{a+b\le m}\frac{p_{a,b,\beta}(w_1)}{w_2^{m-b}}\cdot\frac{\partial^{a+b}}{\partial w_1^a\partial w_2^b},
\end{equation}
where $p_{a,b,\beta}(w_1)$ is a polynomial of degree at most $m$ in variable $w_1$. In addition, if $\abs{\beta}\le k$ for some $k\in\mathbb{Z}^+\cup\{0\}$, then $\abs{p_{a,b,\beta}(w_1)}\le C_k$ on $\mathbb{D}$ uniformly in $\beta$, $a$, and $b$, for some constant $C_k>0$ depending only on $k$.
\end{lemma}

\begin{proof}
We prove \eqref{DztoDw} by induction on $m=\abs{\beta}$. The case $m=0$ is trivial. When $m=1$, a direct computation shows
\[
\frac{\partial}{\partial z_1}=\frac{1}{w_2}\cdot\frac{\partial}{\partial w_1}
\]
and
\[
\frac{\partial}{\partial z_2}=-\frac{w_1}{w_2}\cdot\frac{\partial}{\partial w_1}+\frac{\partial}{\partial w_2}.
\]
So both of $\frac{\partial}{\partial z_1}$ and $\frac{\partial}{\partial z_2}$ are of the form in \eqref{DztoDw}.

Suppose for all $\beta$ with $\abs{\beta}=m$, $D_{z}^{\beta}$'s are of the form in \eqref{DztoDw}, we now check the case $\abs{\beta'}=m+1$. Note that $D_{z}^{\beta'}=\frac{\partial}{\partial z_1}\circ D_{z}^{\beta}$ or $D_{z}^{\beta'}=\frac{\partial}{\partial z_2}\circ D_{z}^{\beta}$ for some $\beta$. By the inductive assumption, we have
\begin{align*}
\frac{\partial}{\partial z_1}\circ D_{z}^{\beta}
&=\frac{1}{w_2}\cdot\frac{\partial}{\partial w_1}\circ\sum_{a+b\le m}\frac{p_{a,b,\beta}(w_1)}{w_2^{m-b}}\cdot\frac{\partial^{a+b}}{\partial w_1^a\partial w_2^b}\\
&=\sum_{a+b\le m}\frac{p'_{a,b,\beta}(w_1)}{w_2^{m+1-b}}\cdot\frac{\partial^{a+b}}{\partial w_1^a\partial w_2^b}+\frac{p_{a,b,\beta}(w_1)}{w_2^{m+1-b}}\cdot\frac{\partial^{a+b+1}}{\partial w_1^{a+1}\partial w_2^b}\\
&=\sum_{a+b\le m+1}\frac{p_{a,b,\beta'}(w_1)}{w_2^{m+1-b}}\cdot\frac{\partial^{a+b}}{\partial w_1^a\partial w_2^b},
\end{align*}
and
\begin{align*}
\frac{\partial}{\partial z_2}\circ D_{z}^{\beta}
&=\left(-\frac{w_1}{w_2}\cdot\frac{\partial}{\partial w_1}+\frac{\partial}{\partial w_2}\right)\circ\sum_{a+b\le m}\frac{p_{a,b,\beta}(w_1)}{w_2^{m-b}}\cdot\frac{\partial^{a+b}}{\partial w_1^a\partial w_2^b}\\
&=\sum_{a+b\le m}\frac{-w_1p'_{a,b,\beta}(w_1)}{w_2^{m+1-b}}\cdot\frac{\partial^{a+b}}{\partial w_1^a\partial w_2^b}+\frac{-w_1p_{a,b,\beta}(w_1)}{w_2^{m+1-b}}\cdot\frac{\partial^{a+b+1}}{\partial w_1^{a+1}\partial w_2^b}+\\
&\,\,\,\,\,\,\,\,\,\,\,\,\,\,\,\,\,\,\,\,\,\,\,\,\,\frac{(b-m)p_{a,b,\beta}(w_1)}{w_2^{m+1-b}}\cdot\frac{\partial^{a+b}}{\partial w_1^a\partial w_2^b}+\frac{p_{a,b,\beta}(w_1)}{w_2^{m-b}}\cdot\frac{\partial^{a+b+1}}{\partial w_1^a\partial w_2^{b+1}}\\
&=\sum_{a+b\le m+1}\frac{p_{a,b,\beta'}(w_1)}{w_2^{m+1-b}}\cdot\frac{\partial^{a+b}}{\partial w_1^a\partial w_2^b}.
\end{align*}
We see that $p_{a,b,\beta'}(w_1)$ is a polynomial of degree at most $m+1$ and $D_{z}^{\beta'}$ has the form in \eqref{DztoDw}.

When $\abs{\beta}\le k$, all the possible combinations of derivatives in $D_{z}^{\beta}$ are finite. So there are finitely many different coefficients in all of the $p_{a,b,\beta}(w_1)$'s. Since $\abs{w_1}\le1$ on $\mathbb{D}$ and $a,b\le m\le k$, we obtain $\abs{p_{a,b,\beta}(w_1)}\le C_k$ on $\mathbb{D}$ as desired.
\end{proof}

Now we can transfer $\mathbb{H}$ to the product model $\mathbb{D}\times\mathbb{D}^*$ by the biholomorphism $\Phi$. Combining \eqref{trans} and \eqref{DztoDw}, we see that the right hand side of \eqref{eq1} becomes

\begin{equation}
\label{eq2}
\sum_{\abs{\beta}\le k}\int_{\mathbb{D}\times\mathbb{D}^*}\Big|\sum_{a+b\le\abs{\beta}}\int_{\mathbb{D}\times\mathbb{D}^*}K_{a,b,\beta}(w,\eta)f(\Psi(\eta))\abs{\eta_2}^2\,d\eta\Big|^p\abs{w_2}^{kp+2}\,dw,
\end{equation}
where
\[
K_{a,b,\beta}(w,\eta)=\frac{p_{a,b,\beta}(w_1)}{w_2^{\abs{\beta}-b}}\cdot\frac{\partial^a}{\partial w_1^a}\Big(\frac{1}{(1-w_1\overline{\eta}_1)^2}\Big)\cdot\frac{\partial^b}{\partial w_2^b}\Big(\frac{1}{w_2\overline{\eta}_2}\cdot\frac{1}{(1-w_2\overline{\eta}_2)^2}\Big).
\]

\section{Convert the Differential Operators on $\mathbb{D}^*$}
\label{sec3}

\subsection{Convert to the Anti-holomorphic Part}

Since $\mathbb{D}^*$ is a Reinhardt domain, by using the idea in \cite{Str1}, we can convert the differential operators as follows.

\begin{lemma}
As in \eqref{eq2}, for the last factor in $K_{a,b,\beta}(w,\eta)$, we have
\begin{equation}
\label{DwtoDeta}
\frac{\partial^b}{\partial w_2^b}\Big(\frac{1}{w_2\overline{\eta}_2}\cdot\frac{1}{(1-w_2\overline{\eta}_2)^2}\Big)=\frac{\overline{\eta}_2^b}{w_2^b}\cdot\frac{\partial^b}{\partial\overline{\eta}_2^b}\Big(\frac{1}{w_2\overline{\eta}_2}\cdot\frac{1}{(1-w_2\overline{\eta}_2)^2}\Big).
\end{equation}
\end{lemma}

\begin{proof}
The kernel in \eqref{DwtoDeta} is the weighted Bergman kernel associated to $\mathbb{D}^*$ with the weight $\abs{z}^2$, see \cite{C2}. It has the following expansion
\[
\frac{1}{w_2\overline{\eta}_2}\cdot\frac{1}{(1-w_2\overline{\eta}_2)^2}=\sum_{j=0}^{\infty}(j+1)(w_2\overline{\eta}_2)^{j-1}
\]
which converges uniformly on every compact subset $K\times K\subset\mathbb{D}^*\times\mathbb{D}^*$. Differentiate the series term by term, and we see that
\begin{align*}
w_2^b\cdot\frac{\partial^b}{\partial w_2^b}\Big(\frac{1}{w_2\overline{\eta}_2}\cdot\frac{1}{(1-w_2\overline{\eta}_2)^2}\Big)
&=\sum_{j=0}^{\infty}(j+1)w_2^b\cdot\frac{\partial^b}{\partial w_2^b}(w_2\overline{\eta}_2)^{j-1}\\
&=\sum_{j=0}^{\infty}(j+1)\overline{\eta}_2^b\cdot\frac{\partial^b}{\partial\overline{\eta}_2^b}(w_2\overline{\eta}_2)^{j-1}\\
&=\overline{\eta}_2^b\cdot\frac{\partial^b}{\partial\overline{\eta}_2^b}\Big(\frac{1}{w_2\overline{\eta}_2}\cdot\frac{1}{(1-w_2\overline{\eta}_2)^2}\Big).
\end{align*}
This completes the proof.
\end{proof}

\subsection{Integration by Parts}

Now we focus on the integration over $\mathbb{D}^*$ in \eqref{eq2}. We first define a ``tangential" operator.

\begin{definition}
Let $S_w=w\frac{\partial}{\partial w}$ be the \textit{complex normal differential operator} on a neighborhood of $\partial\mathbb{D}$. We define the \textit{tangential operator} by
\[
T_w=\Im(S_w)=\frac{1}{2i}\left(w\frac{\partial}{\partial w}-\overline{w}\frac{\partial}{\partial\overline{w}}\right).
\]
\end{definition}

\begin{remark}
Indeed, $T_w$ is well defined on a neighborhood of $\overline{\mathbb{D}}$. Moreover, for any disc $\mathbb{D}_{\rho}=\{\abs{w}<\rho\}$ of radius $\rho<1$ with defining function $r_{\rho}(w)=\abs{w}^2-\rho^2$, we have
\begin{equation}
\label{tangential}
T_w(r_{\rho})=0
\end{equation}
on $\partial\mathbb{D}_{\rho}$. That is, $T_w$ is tangential on $\partial\mathbb{D}_{\rho}$ for all $\rho<1$.
\end{remark}

In order to make use of integration by parts, we need the following lemma.

\begin{lemma}
Let $T_w$ be as above, for $b\in\mathbb{Z}^+\cup\{0\}$, we have
\begin{equation}
\label{DtoT}
T_w^b\equiv\sum_{j=0}^bc_j\overline{w}^j\frac{\partial^j}{\partial\overline{w}^j}\,\,\,\,\,\pmod{\frac{\partial}{\partial w}},
\end{equation}
where $c_j$'s are constants, $c_b\neq0$, and $T_w^b$ denotes the composition of $b$ copies of $T_w$.
\end{lemma}

\begin{proof}
We prove \eqref{DtoT} by induction on $b$. The case $b=0$ is trivial. When $b=1$, it is easy to see that
\[
T_w\equiv-\frac{1}{2i}\overline{w}\frac{\partial}{\partial\overline{w}}\,\,\,\,\,\pmod{\frac{\partial}{\partial w}}.
\]

Suppose \eqref{DtoT} holds for some $b$, then we see that
\[
T_w^b=\sum_{j=0}^bc_j\overline{w}^j\frac{\partial^j}{\partial\overline{w}^j}+A\circ\frac{\partial}{\partial w},
\]
for some operator $A$. So for the case $b+1$, we have
\begin{align*}
T_w\circ T_w^b
&=\frac{1}{2i}\left(w\frac{\partial}{\partial w}-\overline{w}\frac{\partial}{\partial\overline{w}}\right)\circ\bigg(\sum_{j=0}^bc_j\overline{w}^j\frac{\partial^j}{\partial\overline{w}^j}+A\circ\frac{\partial}{\partial w}\bigg)\\
&=\frac{1}{2i}\bigg(\sum_{j=0}^bc_jw\overline{w}^j\frac{\partial^j}{\partial\overline{w}^j}\frac{\partial}{\partial w}-jc_j\overline{w}^j\frac{\partial^j}{\partial\overline{w}^j}-c_j\overline{w}^{j+1}\frac{\partial^{j+1}}{\partial\overline{w}^{j+1}}\bigg)+T_w\circ A\circ\frac{\partial}{\partial w}\\
&=\sum_{j=0}^{b+1}c'_j\overline{w}^j\frac{\partial^j}{\partial\overline{w}^j}+A'\circ\frac{\partial}{\partial w},
\end{align*}
for some constants $c'_j$'s with $c'_{b+1}=-\frac{1}{2i}c_b\neq0$ and some operator $A'$. Therefore, \eqref{DtoT} holds for $T_w^{b+1}$.
\end{proof}

Combine \eqref{DwtoDeta} and \eqref{DtoT}, since the kernel in \eqref{DwtoDeta} is anti-holomorphic in $\eta_2$, the inside integration over $\mathbb{D}^*$ w.r.t. variable $\eta_2$ in \eqref{eq2} denoted by $I$ becomes

\begin{align*}
I
&=\int_{\mathbb{D}^*}\frac{\partial^b}{\partial w_2^b}\bigg(\frac{1}{w_2\overline{\eta}_2}\cdot\frac{1}{(1-w_2\overline{\eta}_2)^2}\bigg)f(\Psi(\eta))\abs{\eta_2}^2\,d\eta_2\\
&=\int_{\mathbb{D}^*}\frac{\overline{\eta}_2^b}{w_2^b}\cdot\frac{\partial^b}{\partial\overline{\eta}_2^b}\bigg(\frac{1}{w_2\overline{\eta}_2}\cdot\frac{1}{(1-w_2\overline{\eta}_2)^2}\bigg)f(\Psi(\eta))\abs{\eta_2}^2\,d\eta_2\\
&=\frac{1}{w_2^b}\int_{\mathbb{D}^*}\sum_{j=0}^bc_jT_{\eta_2}^j\bigg(\frac{1}{w_2\overline{\eta}_2}\cdot\frac{1}{(1-w_2\overline{\eta}_2)^2}\bigg)f(\Psi(\eta))\abs{\eta_2}^2\,d\eta_2\\
&=\frac{1}{w_2^b}\sum_{j=0}^bc_j\lim_{\epsilon\to0^+}\int_{\mathbb{D}-\mathbb{D}_{\epsilon}}T_{\eta_2}^j\bigg(\frac{1}{w_2\overline{\eta}_2}\cdot\frac{1}{(1-w_2\overline{\eta}_2)^2}\bigg)f(\Psi(\eta))\abs{\eta_2}^2\,d\eta_2.
\end{align*}

Let us assume in addition for a moment that $f(\Psi(\eta))$ belongs to $C^{\infty}(\overline{\mathbb{D}}-\{0\})$ in variable $\eta_2$. Then by \eqref{tangential} we obtain

\begin{equation}
\label{intbypart}
\begin{split}
I
&=\frac{1}{w_2^b}\sum_{j=0}^bc_j\lim_{\epsilon\to0^+}\int_{\mathbb{D}-\mathbb{D}_{\epsilon}}T_{\eta_2}^j\bigg(\frac{1}{w_2\overline{\eta}_2}\cdot\frac{1}{(1-w_2\overline{\eta}_2)^2}\bigg)f(\Psi(\eta))\abs{\eta_2}^2\,d\eta_2\\
&=\frac{1}{w_2^b}\sum_{j=0}^bc_j(-1)^j\lim_{\epsilon\to0^+}\int_{\mathbb{D}-\mathbb{D}_{\epsilon}}\frac{1}{w_2\overline{\eta}_2}\cdot\frac{1}{(1-w_2\overline{\eta}_2)^2}T_{\eta_2}^j\bigg(f(\Psi(\eta))\abs{\eta_2}^2\bigg)\,d\eta_2\\
&=\frac{1}{w_2^b}\sum_{j=0}^b(-1)^jc_j\int_{\mathbb{D}^*}\frac{1}{w_2\overline{\eta}_2}\cdot\frac{1}{(1-w_2\overline{\eta}_2)^2}T_{\eta_2}^j\Big(f(\Psi(\eta))\Big)\abs{\eta_2}^2\,d\eta_2,
\end{split}
\end{equation}
where the last line follows from the fact that $T_{\eta_2}(\abs{\eta_2}^2)=0$.

\begin{definition}
\label{operators}
We use the following notation
\[
F_j(\eta)=T_{\eta_2}^j\Big(f(\Psi(\eta))\Big)\cdot\eta_2,
\]
\[
\mathcal{B}_{1,a}(g)(w_1)=\int_{\mathbb{D}}\frac{\partial^a}{\partial w_1^a}\bigg(\frac{1}{(1-w_1\overline{\eta}_1)^2}\bigg)g(\eta_1)\,d\eta_1,
\]
for any $g$ whenever the integral is well defined, and
\[
\mathcal{B}_2(h)(w_2)=\int_{\mathbb{D}^*}\frac{h(\eta_2)}{(1-w_2\overline{\eta}_2)^2}\,d\eta_2,
\]
for any $h$ whenever the integral is well defined.
\end{definition}

By \eqref{intbypart} and the notation above (Definition \ref{operators}), we see that \eqref{eq2} becomes

\begin{equation}
\label{eq3}
\sum_{\abs{\beta}\le k}\int_{\mathbb{D}\times\mathbb{D}^*}\bigg|\sum_{a+b\le\abs{\beta}}\frac{p_{a,b,\beta}(w_1)}{w_2^{\abs{\beta}+1}}\sum_{j=0}^b(-1)^jc_j\mathcal{B}_{1,a}\big(\mathcal{B}_2(F_j)\big)(w)\bigg|^p\abs{w_2}^{kp+2}\,dw.
\end{equation}

\section{Proof of the Main Theorem}
\label{sec4}

\subsection{The $L^p$ Boundedness}
To finish the proof, we first need two lemmas.

\begin{lemma}
\label{B1}
The operator $\mathcal{B}_{1,a}$ defined as in Definition \ref{operators} is bounded from $W^{a,p}(\mathbb{D})$ to $L^p(\mathbb{D})$ for $p\in(1,\infty)$.
\end{lemma}

\begin{proof}
This follows from the well-known result that the Bergman projection on $\mathbb{D}$ is bounded from $W^{k,p}(\mathbb{D})$ to itself for $p\in(1,\infty)$ and all $k\in\mathbb{Z}^+\cup\{0\}$. 
\end{proof}

\begin{lemma}
\label{B2}
The integral operator $\mathcal{B}_2$ defined as in Definition \ref{operators} is bounded from $L^p\big(\mathbb{D}^*,\abs{w}^{2-p}\big)$ to itself for $p\in(4/3,4)$, where $L^p\big(\mathbb{D}^*,\abs{w}^{2-p}\big)$ is the weighted $L^p$ space with $w\in\mathbb{D}^*$.
\end{lemma}

\begin{proof}
This is equivalent to the statement that the weighted Bergman projection associated to $\mathbb{D}^*$ with the weight $\abs{w}^2$ is bounded from $L^p\big(\mathbb{D}^*,\abs{w}^{2}\big)$ to itself for $p\in(4/3,4)$. For a proof, see \cite{C2}.
\end{proof}

\subsection{The Proof under the Additional Assumption}

With Lemma \ref{B1} and Lemma \ref{B2}, we can prove Theorem \ref{main} under the additional assumption $f(\Psi(\eta))\in C^{\infty}(\overline{\mathbb{D}}-\{0\})$ in variable $\eta_2$.

\begin{proof}[Proof of Theorem \ref{main} under additional assumption]
\

By \eqref{eq1}, \eqref{eq2}, \eqref{eq3} and Lemma \ref{lem}, we obtain

\begin{align*}
\norm{\mathcal{B}(f)}_{k,p,kp}^p
&\le\sum_{\abs{\beta}\le k}\sum_{a+b\le\abs{\beta}}\sum_{j=0}^bC_{k,p}\int_{\mathbb{D}\times\mathbb{D}^*}\abs{\mathcal{B}_{1,a}(\mathcal{B}_2(F_j))(w)}^p\abs{w_2}^{kp+2-p(\abs{\beta}+1)}\,dw\\
&\le C_{k,p}\sum_{a+b\le k}\int_{\mathbb{D}\times\mathbb{D}^*}\abs{\mathcal{B}_{1,a}(\mathcal{B}_2(F_b))(w)}^p\abs{w_2}^{2-p}\,dw.
\end{align*}
By Lemma \ref{B1}, for $p\in(1,\infty)$ we have

\begin{align*}
\norm{\mathcal{B}(f)}_{k,p,kp}^p
&\le C_{k,p}\sum_{a+b\le k}\int_{\mathbb{D}^*}\bigg(\int_{\mathbb{D}}\sum_{\abs{\beta}\le a}\abs{D^{\beta}_{w_1,\overline{w}_1}(\mathcal{B}_2(F_b))(w)}^p\,dw_1\bigg)\abs{w_2}^{2-p}\,dw_2\\
&\le C_{k,p}\sum_{\abs{\beta}+b\le k}\int_{\mathbb{D}}\bigg(\int_{\mathbb{D}^*}\abs{\mathcal{B}_2(D^{\beta}_{w_1,\overline{w}_1}(F_b))(w)}^p\abs{w_2}^{2-p}\,dw_2\bigg)\,dw_1.
\end{align*}
Similarly, by Lemma \ref{B2}, for $p\in(4/3,4)$ we have

\begin{equation}
\label{eq4}
\begin{split}
\norm{\mathcal{B}(f)}_{k,p,kp}^p
&\le C_{k,p}\sum_{\abs{\beta}+b\le k}\int_{\mathbb{D}}\bigg(\int_{\mathbb{D}^*}\abs{D^{\beta}_{w_1,\overline{w}_1}(F_b)(w)}^p\abs{w_2}^{2-p}\,dw_2\bigg)\,dw_1\\
&=C_{k,p}\sum_{\abs{\beta}+b\le k}\int_{\mathbb{D}\times\mathbb{D}^*}\abs{D^{\beta}_{w_1,\overline{w}_1}T_{w_2}^b\Big(f(\Psi(w))\Big)\cdot w_2}^p\abs{w_2}^{2-p}\,dw\\
&=C_{k,p}\sum_{\abs{\beta}+b\le k}\int_{\mathbb{D}\times\mathbb{D}^*}\abs{D^{\beta}_{w_1,\overline{w}_1}T_{w_2}^b\Big(f(\Psi(w))\Big)}^p\abs{w_2}^{2}\,dw\\
&\le C_{k,p}\sum_{\abs{\beta}+\abs{\beta'}\le k}\int_{\mathbb{D}\times\mathbb{D}^*}\abs{D^{\beta}_{w_1,\overline{w}_1}D^{\beta'}_{w_2,\overline{w}_2}\Big(f(\Psi(w))\Big)}^p\abs{w_2}^{2}\,dw,
\end{split}
\end{equation}
where the last line follows from $T_{w_2}=\frac{1}{2i}\left(w_2\frac{\partial}{\partial w_2}-\overline{w}_2\frac{\partial}{\partial\overline{w}_2}\right)$, $\abs{w_2}<1$ for $w_2\in\mathbb{D}^*$, and a similar equation as \eqref{DtoT}.

By the biholomorphism $\Psi(w)=z$ defined in section \ref{sec2}, we have
\[
\frac{\partial}{\partial w_1}=w_2\frac{\partial}{\partial z_1}\,\,\,\,\,\,\,\,\,\text{and}\,\,\,\,\,\,\,\,\,\frac{\partial}{\partial\overline{w}_1}=\overline{w}_2\frac{\partial}{\partial\overline{z}_1},
\]
and also
\[
\frac{\partial}{\partial w_2}=w_1\frac{\partial}{\partial z_1}+\frac{\partial}{\partial z_2}\,\,\,\,\,\,\,\,\,\text{and}\,\,\,\,\,\,\,\,\,\frac{\partial}{\partial\overline{w}_2}=\overline{w}_1\frac{\partial}{\partial\overline{z}_1}+\frac{\partial}{\partial\overline{z}_2}.
\]
Again, since $(w_1,w_2)\in\mathbb{D}\times\mathbb{D}^*$, we have $\abs{w_1},\abs{w_2}<1$. Therefore, by \eqref{eq4} and transferring $\mathbb{D}\times\mathbb{D}^*$ back to $\mathbb{H}$, we finally arrive at
\[
\norm{\mathcal{B}(f)}_{k,p,kp}^p\le C_{k,p}\sum_{\abs{\alpha}\le k}\int_{\mathbb{H}}\abs{D^{\alpha}_{z,\overline{z}}(f)(z)}^p\,dz
\]
as desired.
\end{proof}

\subsection{Remove the Additional Assumption}

To remove the additional assumption $f(\Psi(\eta))\in C^{\infty}(\overline{\mathbb{D}}-\{0\})$ in variable $\eta_2$, we need the following lemma.

\begin{lemma}
\label{smooth}
The subspace $C^{\infty}(\overline{\mathbb{D}}-\{0\})\bigcap W^{k,p}(\mathbb{D}^*,\abs{w}^2)$ is dense in $W^{k,p}(\mathbb{D}^*,\abs{w}^2)$ w.r.t. the weighted norm in $W^{k,p}(\mathbb{D}^*,\abs{w}^2)$.
\end{lemma}

\begin{proof}
The argument is based on \cite[\S5.3 Theorem 2 and Theorem 3]{E}.

Given any $g\in W^{k,p}(\mathbb{D}^*,\abs{w}^2)$, fix $\varepsilon>0$. On $V_0=\mathbb{D}-\overline{\mathbb{D}_{\frac{1}{2}}}$, the weighted norm $W^{k,p}(V_0,\abs{w}^2)$ is equivalent to the unweighted norm $W^{k,p}(V_0)$. Arguing as in the proof of \cite[\S5.3 Theorem 3]{E}, we see that there is a $g_0\in C^{\infty}(\overline{V_0})$, so that
\[
\norm{g_0-g}_{W^{k,p}(V_0,\abs{w}^2)}<\varepsilon.
\]
Define $U_j=\mathbb{D}_{\rho-\frac{1}{j}}-\overline{\mathbb{D}_{\frac{1}{j}}}$ for some $1>\rho>\frac{1}{2}$ and for $j\in\mathbb{Z}^+$ ($U_1=\varnothing$). Let $V_j=U_{j+3}-\overline{U_{j+1}}$, then we see $\bigcup_{j=1}^{\infty}V_j=\mathbb{D}_{\rho}-\{0\}$. Arguing as in the proof of \cite[\S 5.3 Theorem 2]{E}, we can find a smooth partition of unity $\{\psi_j\}_{j=1}^{\infty}$ subordinate to $\{V_j\}_{j=1}^{\infty}$, so that $\sum_{j=1}^{\infty}\psi_j=1$ on $\mathbb{D}_{\rho}-\{0\}$. Moreover, for each $j$, the support of $\psi_jg$ lies in $V_j$ (so $\abs{w}>\frac{1}{j+3}$), and hence $\psi_jg\in W^{k,p}(\mathbb{D}_{\rho}-\{0\})$. Therefore, we can find smooth function $g_j$ with support in $U_{j+4}-\overline{U_j}$, so that
\[
\norm{g_j-\psi_jg}_{W^{k,p}(\mathbb{D}_{\rho}-\{0\})}\le\frac{\varepsilon}{2^j},
\]
see \cite[\S 5.3 Theorem 2]{E} for details. Write $\tilde{g}_0=\sum_{j=1}^{\infty}g_j$, it is easy to see that $\tilde{g}_0\in C^{\infty}(\mathbb{D}_{\rho}-\{0\})$ and
\[
\norm{\tilde{g}_0-g}_{W^{k,p}(\mathbb{D}_{\rho}-\{0\},\abs{w}^2)}\le\norm{\tilde{g}_0-g}_{W^{k,p}(\mathbb{D}_{\rho}-\{0\})}\le\varepsilon,
\]
since $\abs{w}<1$ on $\mathbb{D}_{\rho}-\{0\}$.

Let $V'_0$ be an open set so that $\partial\mathbb{D}\subset V'_0$ and $V'_0\bigcap\mathbb{D}=V_0$, then $V'_0\bigcup\mathbb{D}_{\rho}$ cover $\overline{\mathbb{D}}$. Take a smooth partition of unity $\{\tilde{\psi}_1,\tilde{\psi}_2\}$ on $\overline{\mathbb{D}}$ subordinate to $\{V'_0,\mathbb{D}_{\rho}\}$, then $h=\tilde{\psi}_1g_0+\tilde{\psi}_2\tilde{g}_0$ belongs to $C^{\infty}(\overline{\mathbb{D}}-\{0\})$, and
\begin{align*}
\norm{h-g}_{W^{k,p}(\mathbb{D}^*,\abs{w}^2)}
&\le C\big(\norm{g_0-g}_{W^{k,p}(V_0,\abs{w}^2)}+\norm{\tilde{g}_0-g}_{W^{k,p}(\mathbb{D}_{\rho}-\{0\},\abs{w}^2)}\big)\\
&<2C\varepsilon
\end{align*}
as desired.
\end{proof}

Now we are ready to remove the extra assumption and prove our main result.

\begin{proof}[Proof of Theorem \ref{main}]
\

For any $f\in W^{k,p}(\mathbb{H})$, we have $f(\Psi(w))\in W^{k,p}(\mathbb{D}^*,\abs{w_2}^2)$ in variable $w_2$. Then by Lemma \ref{smooth}, we can find a sequence $\{h_{j}(w)\}\subset C^{\infty}(\overline{\mathbb{D}}-\{0\})$ tending to $f(\Psi(w))$ in variable $w_2$ w.r.t. the norm in $ W^{k,p}(\mathbb{D}^*,\abs{w_2}^2)$. We have already seen that \eqref{eq4} holds for each $h_j(w)$ replacing $f(\Psi(w))$. Indeed, if we focus on the integration over $\mathbb{D}^*$, by comparing with \eqref{eq2}, we see that \eqref{eq4} is just the following: for each $b=0,1,\dots,k$
\begin{equation}
\label{B3}
\int_{\mathbb{D}^*}\abs{w_2^b\frac{\partial^b}{\partial w_2^b}(\mathcal{B}_3(h_j))}^p\abs{w_2}^2\,dw_2\le C_{k,p}\norm{h_j}_{W^{k,p}(\mathbb{D}^*,\abs{w_2}^2)},
\end{equation}
where $\mathcal{B}_3$ is the weighted Bergman projection associated to $\mathbb{D}^*$ with the weight $\abs{w_2}^2$.

Now letting $j\to\infty$, in view of the boundedness of $\mathcal{B}_3$ (Lemma \ref{B2}), we see that $w_2^b\frac{\partial^b}{\partial w_2^b}(\mathcal{B}_3(h_j))$ indeed tends to $w_2^b\frac{\partial^b}{\partial w_2^b}(\mathcal{B}_3(f(\Psi)))$ in $ L^p(\mathbb{D}^*,\abs{w_2}^2)$ for each $b=0,1,\dots,k$. Therefore, \eqref{B3} is valid for general $f(\Psi(w))\in W^{k,p}(\mathbb{D}^*,\abs{w_2}^2)$, which completes the proof for any general $f\in W^{k,p}(\mathbb{H})$.
\end{proof}

\begin{remark}
The method also applies to the $n$-dimensional generalization of the Hartogs triangle, see \cite{C1}. To be precise, for $j=1,\dots,l$, let $\Omega_j$ be a bounded smooth domain in $\mathbb{C}^{m_j}$ with a biholomorphic mapping $\phi_j:\Omega_j\to\mathbb{B}^{m_j}$ between $\Omega_j$ and the unit ball $\mathbb{B}^{m_j}$ in $\mathbb{C}^{m_j}$. We use the notation $\tilde{z_j}$ to denote the $j$th $m_j$-tuple in $z\in\mathbb{C}^{m_1+\cdots+m_l}$, that is $z=(\tilde{z_1},\dots,\tilde{z_l})$. Let $n=m_1+\cdots+m_l+n'$, $n-n'\ge1$, and $n'\ge1$, we define the \textit{$n$-dimensional Hartogs triangle} by 
\[
\mathbb{H}_{\phi_j}^n=\left\{(z,z')\in\mathbb{C}^{m_1+\cdots+m_l+n'}\,:\,\max_{1\le j\le l}\abs{\phi_j(\tilde{z_j})}<\abs{z'_1}<\abs{z'_2}<\cdots<\abs{z'_{n'}}<1\right\}.
\]

Following the same idea, we see that the Bergman projection $\mathcal{B}$ on $\mathbb{H}_{\phi_j}^n$ is bounded from $W^{k,p}(\mathbb{H}_{\phi_j}^n)$ to $W^{k,p}(\mathbb{H}_{\phi_j}^n,\abs{z'_1}^{kp})$ for $p\in(\frac{2n}{n+1},\frac{2n}{n-1})$. However, the weight $\abs{z'_1}$ is no longer comparable to $\abs{(z,z')}$, the distance from points near the boundary to the singularity at the boundary.
\end{remark}

\bibliographystyle{plain}

\begin{thebibliography}{10000000}

\bibitem[Che13]{C1} L. Chen, The $L^p$ boundedness of the Bergman projection for a class of bounded Hartogs domains, arXiv:1304.7898 [math.CV] (Apr. 2013).

\bibitem[Che15a]{C2} L. Chen, Weighted Bergman projection on the Hartogs triangle, arXiv:1410.6205v2 [math.CV] (Apr. 2015).

\bibitem[Che15b]{Chen} L. Chen, Regularity of the Bergman projection on variants of the Hartogs triangle, Ph.D. Thesis, Washington University in St. Louis, 2015.

\bibitem[CS13]{CS2} D. Chakrabarti and M.-C. Shaw, Sobolev regularity of the $\overline{\partial}$-equation on the Hartogs triangle, \textit{Math. Ann.} \textbf{356} no. 1 (2013), 241--258.

\bibitem[CZ15]{CZ} D. Chakrabarti and Y. E. Zeytuncu, $L^p$ mapping properties of the Bergman projection on the Hartogs triangle, \textit{Proc. Amer. Math. Soc.}, to appear.

\bibitem[Eva98]{E} L. C. Evans, \textit{Partial Differential Equations}, Amer. Math. Soc., Providence, RI, 1998. Third printing, 2002.

\bibitem[Kra01]{K} S. Krantz, \textit{Function theory of several complex variables}, 2nd ed., Amer. Math. Soc., Providence, RI, 2001.

\bibitem[KP08]{KP} S. Krantz and M. Peloso, The Bergman kernel and projection on non-smooth worm domains, \textit{Houston J. Math.} \textbf{34} (2008), 873--950.

\bibitem[LS04]{LS1} L. Lanzani and E. M. Stein, Szeg\H{o} and Bergman projections on non-smooth planar domains, \textit{J. Geom. Anal.} \textbf{14} no. 1 (2004), 63--86.

\bibitem[LS12]{LS2} L. Lanzani and E. M. Stein, The Bergman projection in $L^p$ for domains with minimal smoothness, \textit{Illinois J. Math.} \textbf{56} no. 1 (2012), 127--154.

\bibitem[Str86]{Str1} E. Straube, Exact regularity of Bergman, Szeg\H{o} and Sobolev space projections in non-pseudoconvex domains, \textit{Math. Z.} \textbf{192} (1986), 117--128.

\bibitem[Zey13]{Z} Y. E. Zeytuncu, $L^p$ regularity of weighted Bergman projections, \textit{Trans. Amer. Math. Soc.} \textbf{365} no. 6 (2013), 2959--2976.

\end{thebibliography}

\end{document}